\documentclass[twocolumn]{article}
\pdfoutput=1

\newcommand{\be}{\begin{equation}}
\newcommand{\ee}{\end{equation}}

\usepackage[margin=2cm]{geometry}

\usepackage{bm}
\usepackage{graphicx} 
\usepackage{epstopdf}
\usepackage[linesnumbered,ruled,vlined]{algorithm2e}

\usepackage{enumitem,calc}
\usepackage[colorlinks=false, hidelinks]{hyperref}
\hypersetup{pdfauthor=author}
\usepackage{subfigure}
\usepackage{tikz,pgfplots}
\usepackage{tkz-euclide}
\usepackage{datatool} 
\usepackage{calc}
\usepackage{xsavebox} 
\usetikzlibrary{3d,positioning}
\pgfplotsset{compat = newest}
\usetikzlibrary{arrows,calc,positioning,shadows,shapes,fit}
\usetikzlibrary{pgfplots.statistics, pgfplots.colorbrewer}

\usepackage{amsmath,amssymb,amsthm}
\usepackage{mathrsfs}
\usepackage{verbatim}
\usepackage{balance}

\newtheorem{remark}{Remark}
\newtheorem{definition}{Definition}
\newtheorem{theorem}{Theorem}
\newtheorem{lemma}{Lemma}

\def\mytmpindentfour{4.0ex}

\let\today\relax
\makeatletter
\def\ps@pprintTitle{%
    \let\@oddhead\@empty
    \let\@evenhead\@empty
    \def\@oddfoot{\footnotesize\itshape
         {~}
         \hfill\today}%
    \let\@evenfoot\@oddfoot
    }
\makeatother

\makeatletter 
\def\l@section#1#2{%
  \ifnum \c@tocdepth > \m@ne
    \addpenalty{-\@highpenalty}%
    \vskip 0.3em \@plus\p@
    \setlength\@tempdima{2.5em}%
    \begingroup
      \parindent \z@ \rightskip \@pnumwidth
      \parfillskip -\@pnumwidth
      \leavevmode \normalfont
      \advance\leftskip\@tempdima
      \hskip -\leftskip
      {#1}\nobreak
      \hfil \nobreak
      \hb@xt@\@pnumwidth{\hss #2}\par
    \endgroup
  \fi} 
\makeatother

\begin{document}

\title{\bf \Large An Evolution Process for Effective Network Topological Compression}

\author{Jian-Hui Li$^{1,2,3}$, Zu-Guo Yu$^{1, 2}$\thanks{Corresponding author, email: yuzg@xtu.edu.cn}, and Yu-Chu Tian$^{3}$\thanks{Corresponding author, email: y.tian@qut.edu.au} \\
{\small $^{1}$National Center for Applied Mathematics in Hunan, Xiangtan University, Hunan 411105, China.}\\
{\small $^{2}$Key Laboratory of Intelligent Computing and Information Processing of the Ministry of Education}\\
{\small  of China, Xiangtan University, Hunan 411105, China. }\\
{\small $^{3}$School of Computer Sciences, Queensland University of Technology, GPO Box 2434,}\\
{\small Brisbane, QLD 4001, Australia.}
}
\date{}
\maketitle

\begin{abstract}
Network dynamics offers critical insights into the behavior and evolution of complex systems. Here, we focus on the topological dynamics of networks to explore a unique process for reducing the average distance: topological compression. The compression process essentially involves a series of network topological transformations, which can generally be achieved through rewiring technique. This paper proposes an evolutionary mechanism for achieving effective network topological compression and experimentally validates its performance across various synthetic networks. These results establish a paradigm in the field of network topological compression dynamics.
\end{abstract}


\section{Introduction}

~~~~Complex networks not only hold significant potential for diverse applications but also play a crucial role in advancing interdisciplinary research. Despite significant advances in the field, a comprehensive understanding of network structure and dynamics remains incomplete, highlighting the requirement for further investigation into key issues~\cite{boguna2021network,SSZ2023}. 
Consequently, investigations from multiple perspectives are necessary. Among them, studying node dynamics~\cite{ghavasieh2024diversity,barzon2024unraveling} provides an effective approach to unraveling the underlying mechanisms.

Additionally, the dynamics of network topology are also crucial for understanding complex networks. Various dynamical transformations can reduce the average distance ($\bar{D}$) of networks, which we refer to as $\bar{D}$-reducing transformations. As a foundational model, the Watts-Strogatz (WS) model~\cite{WS1998} generates networks transitioning from regular to random or small-world structures controlled by the rewiring probability $P_{rew}$. A small-world network emerges when $P_{rew}$ takes on relatively small values. The transformation in the WS network model leads to a qualitative change in network topology, affecting both the network type and its degree distribution. Ref.~\cite{soriano2018smart} also proposed a $\bar{D}$-reducing transformation for a network by combining rewiring and edge addition. This process can induce fundamental modification of the network topology, as the network may even become fully connected when a sufficient number of edges are added.

Network transformations that preserve the ensemble to which a network belongs, as defined in previous work~\cite{doyle2005The}, can offer valuable insights into network structure without altering its statistical class.
Previous studies~\cite{gallos2012Asmall,doyle2005The,ye2013multi} have introduced network topological transformations that preserve the degree distribution of the network. These approaches, based on network rewiring technology, achieve $\bar{D}$-reduction in the network.

We investigate a new type of network topological transformation process\textemdash topological compression. The topological compression is abstracted as a dynamic process to tighten them with invariance in the involved nodes, the number of edges, network conncetivity, and degree distribution. In other words, topological compression is a $\bar{D}$-reduction transformation that preserves the network ensemble. Although the random rewiring method proposed in Ref.~\cite{LYA2023} can achieve topology compression of numerous networks. However, the method can be regarded as a preliminary topology compression process, and its effectiveness requires further refinement.

To develop a more efficient compression process, this paper investigates an effective compression evolution of network topology. A notable limitation of general rewiring approaches is their limited effectiveness in the topological compression of specialized network. Intuitively, this is because existing methods are less effective at compressing scale-free networks. This is a further significant motivation for investigating effective topological compression: the goal is to achieve maximal network compression, particularly in challenging cases such as scale-free networks.

The remainder of the paper is structured as follows. Section II covers the theoretical foundations of topological compression process. 
Section III discusses the algorithmic details and computational complexity.
Section IV presents the experimental results.  The discussion and conclusions contain in Section V.

\section{Theoretical development of network topological compression}

Throughout our theoretical developments, we consider complex networks that can be described by a connected simple graph. 
A \emph{simple graph} is an unweighted and undirected graph that has neither multiple edges between any two vertices nor self-loops for any vertex.
Denote $G(V,E)$ as a connected simple graph, where $V$ and $E$ are the sets of vertices and edges of $G$, respectively. Let $|V|$  and $|E|$ represent the sizes of $V$ and $E$, respectively. The degree of vertex $v\in V$ is denoted by $k(v)$. We also denote $(v_i,v_j)$ and $e(v_i,v_j)$ as a pair of vertices and the edge that joins the two vertices, respectively. For simplicity, $e(\cdot,\cdot)$ is denoted as $e$ when no confusion arises. Let $G-e$ and $G+e$ represent the graphs after an edge $e$ is removed and added, respectively. Furthermore, use $d_G(v_i,v_j)$ and $B_G(e)$ to represent the distance of vertex pair $(v_i,v_j)$ and the edge betweenness of edge $e$ in $G$, respectively. As the discussions for graph are for complex networks, the terms ``vertex" and ``nodes" are used interchangeably. 

\subsection{Preliminaries}

Let us discuss some preliminaries below, which are followed by detailed theoretical developments towards network compression. 


\begin{definition}[Path, Ref. \cite{REIN2017}] 
A path is a non-empty graph $P(V_P,E_P)$ of the form 
\begin{align*}
\left\{%
    \begin{array}{rcl}
    V_P &=& \{v_1, \cdots, v_{|V_P|}\},\\
    E_P &=& \{e(v_1,v_2),\cdots, e(v_{|V_P|-1},v_{|V_P|})\},
    \end{array}
\right.
\end{align*}
where nodes in $V_P$ are all distinct. It is often denoted as a sequence of vertices, e.g., $v_1\cdots v_{|V_P|}$.  
\label{def:path}
\end{definition}

\begin{lemma}[Ref. \cite{REIN2017}] 
Let $G(V,E)$ be a simple graph, then $\sum_{v \in V} k(v) = 2|E|$.
\label{lem:sum_k(v)}
\end{lemma}

From Definition \ref{def:path} and Lemma \ref{lem:sum_k(v)}, we can derive the following lemma.

\begin{lemma}[Path subgraph] Let $G(V,E)$ be a simple graph. For $V_P \subseteq V$ and $E_P \subseteq E$, a subgraph $G_P(V_P, E_P)$ of $G$ is a path if and only if the following conditions hold:
 (i) $G_P(V_P,E_P)$ is a connected graph, 
 (ii) $|V_P| = |E_P|+1$, and
(iii) $k(v_i) \leq 2~\forall v_i \in V_P$. 
\label{lem:path_subgraph}
\end{lemma}

\begin{proof} 
If $G_P(V_P,E_P)$ is a path, then the length of the path is $|E_P|$ and $G_P$ can be represented as a sequence of distinct nodes $v_1v_2\cdots v_{|V_P|}$ according to Definition \ref{def:path}. Clearly, conditions (i) through (iii) in Lemma~\ref{lem:path_subgraph} hold. 

Conversely, if conditions (i) through (iii) hold, it follows from condition (i) that $k(v_i) \geq 1~\forall v_i \in V_P$. From condition (iii), we have either $k(v_i)=1$ or $k(v_i)=2~\forall v_i \in V_P$. Let $q$ be the number of nodes with degree 1.  According to Lemma~\ref{lem:sum_k(v)} and condition (ii), we have 
\begin{equation}
\sum_{v\in V_P} k(v) = 2 |V_P| - q = 2 |V_P| - 2.
\label{eqn:sumv}
\end{equation}
This indicates that $q=2$. Let $V_P = \{v_{P,1}$, $v_{P,2}$, $\cdots$, $v_{P,|V_P|}\}$. Without loss of generality, assume $k(v_{P,1}) = k(v_{P,2}) = 1$. Thus, $k(v)=2$ $\forall v \in V_P - \{v_{P,1}, v_{P,2}\} \triangleq V_P^{*}$. 

Since $k(v_{P,1})=1$, there exists only one node $v_3^{*} \in V_P^{*}$ such that $e(v_1, v_3^{*}) \in E_P$. As $k(v_3^{*})=2$, there exists only one node $v_4^{*} \in V_P^{*}-\{v_3^{*}\}$ such that $e(v_3^{*}, v_4^{*}) \in E_P$. Finally, iterating over all nodes and edges of $G_P(V_P, E_P)$, we can get a sequence of nodes  $v_{P,1}v_3^{*}v_4^{*}\cdots v_{|V_P|}^{*}v_{P,2}$ with any two adjacent nodes being connected with an edge that belongs to $E_P$. Then, $G_P(V_P,E_P)$ is a path. 
\end{proof}

\begin{definition}[Distance] For a connected simple graph $G(V,E)$, a path $G_{P,lm}(V_P, E_P)$ between a pair of distinct nodes $v_l\in V$ and $v_m\in V$ is the shortest path between $v_l$ and $v_m$ if its length is minimal in all paths between $v_l$ and $v_m$. The length of the shortest path of the node pair is referred to as the geodesic distance, which we denote as $d_G(v_l,v_m)$. 
\label{def:2}
\end{definition}

\begin{definition}[Path group] 
\label{def:3}
For a connected simple graph $G(V,E)$ and a node pair $(v_l,v_m)$, let $PG_{lm}=\{G_{P,lm}^1(V_1,E_1)$ $,\cdots,G_{P,lm}^g(V_g,E_g)\}$. 
\begin{enumerate} 
\renewcommand{\labelenumi}{(\roman{enumi})}
\item Call $PG_{lm}$ the path group of $(v_l,v_m)$ if  all elements in $PG_{lm}$ are paths of $(v_l, v_m)$. 

\item Call $PG_{lm}$ the maximum path group of $(v_l,v_m)$ if $PG_{lm}$ is the path group of $(v_l,v_m)$ and all paths of $(v_l,v_m)$ in $G$ belong to $PG_{lm}$.

\item Call $PG_{lm}$ the shortest path group of $(v_l,v_m)$ if all elements of $PG_{lm}$ are the shortest paths of $(v_l, v_m)$. 

\item Call $PG_{lm}$ the maximum shortest path group of $(v_l,v_m)$ if $PG_{lm}$ is the shortest path group of $(v_l,v_m)$ and all shortest paths of $(v_l,v_m)$ in $G$ belong to $PG_{lm}$. 
\end{enumerate}
\end{definition}

\begin{remark}
 If $PG_{lm}$ is the maximum shortest path group of $(v_l,v_m)$, then  $|PG_{lm}|$ is equal to the number of shortest paths of $(v_l,v_m)$ in $G$. For a connected simple graph $G$, we have $|PG_{lm}| \geq 1$. For any $G_{P,lm}^{*}(V^{*}, E^{*}) \in PG_{lm}$, we have $d_G(v_l,v_m) = |E^{*}|=|V^{*}|-1$.   
\label{rmk:1}
 \end{remark}

\begin{remark} 
Let $PG_{lm}$, $MPG_{lm}$, $SPG_{lm}$ and $MSPG_{lm}$ denote the path group, the maximum path group, the shortest path group, and the maximum shortest path group of $(v_l,v_m)$, respectively. Then, we have $SPG_{lm} \subseteq MSPG_{lm} \subseteq MPG_{lm}$ and $PG_{lm} \subseteq MPG_{lm}$.  
\label{rmk:2}
\end{remark}

\begin{remark} 
\label{rmk:3}
For a connected simple graph $G(V,E)$, 
$\{MSPG_{lm}\}_{l,m\in\{1,\cdots,|V|\},l\neq m}$ is a set of all shortest path groups of all node pairs in $G(V,E)$. It gives the information about all shortest paths of all node pairs in $G$. For any $l,m\in\{1,\cdots,|V|\},l\neq m$, there exist $G_{P,lm}^{*}(V_{P,lm}^{*}, E_{P,lm}^{*}) \in MSPG_{lm}$.  Thus, the average distance of all node pairs in $G$ is 
\begin{eqnarray}
\bar{D} &=& \frac{1}{|V|(|V|-1)}\sum_{\substack{v_l,v_m\in V, l\neq m}} |E_{P,lm}^{*}| \nonumber\\
&=&\frac{2}{|V|(|V|-1)}\sum_{l=1}^{|V|-1}\sum_{m=l+1}^{|V|} |E_{P,lm}^{*}|.
\end{eqnarray}
\end{remark}

\subsection{Results of graph theory on edge removal and addition}

We investigate the impact of removing an existing edge from, and adding a new edge to, a connected simple graph on the average distance of the graph. More specifically, for a connected simple graph $G(V,E)$, we need to answer the following two questions:
\begin{enumerate}
\renewcommand{\labelenumi}{(\roman{enumi})}
\item How does removing an existing edge $e(a,b)\in E$ affect the average distance $\bar{D}$ of $G$?

\item How does adding a new non-duplicate edge $e$ between a pair of existing nodes affect the average distance $\bar{D}$ of $G$?  
\end{enumerate}

To answer these two questions, we give the following lemmas and remarks. 

\begin{lemma}[Edge removal] 
\label{lem:edge_removal}
Let $G(V,E)$ denote a connected simple graph.  Assume $e(a,b) \in E$ is an edge in $G$. For any node pair $(v_l, v_m)$ in $G$, $v_l\neq v_m$, let $\Delta d(v_l, v_m)$ be the distance variation of the node pair $(v_l, v_m)$ after edge $e$ is removed from $G$. Then, we have 
\begin{equation}
\Delta d(v_l, v_m) = 0 \text{~or~} \Delta d(v_l, v_m) \geq 1. 
\label{eqn:Deltad_cut}
\end{equation}
\end{lemma}

\begin{proof} 
Assume $MSPG_{lm}$ is the maximum shortest path group of $(v_l,v_m)$ in graph $G(V,E)$. According to Remark~\ref{rmk:3}, $\{MSPG_{lm}\}_{l,m\in\{1,\cdots,|V|\},l\neq m}$ gives all shortest paths of all nodes pairs of $G(V,E)$. Consider edge $e(a,b) \in E$ and node pair $(v_l,v_m)~\forall v_l, v_m\in V$ and $l\neq m$. There are the following three cases:

(i) The edge $e(a,b)$ not belong to any element of $MSPG_{lm}$. Firstly, $d_G(v_l,v_m)$ does not increase after the removal of $e(a,b)$ from $G$ since $MSPG_{lm}$ does not change. This means that the original shortest paths of $(v_l,v_m)$ still exist in $G-e(a,b)$. Secondly,  $d(v_l,v_m)$ does not decrease, i.e., there does not exist a shorter path of $(v_l,v_m)$ in $G-e(a,b)$. Let us assume that there exists a path $G_{P,lm}^{*}(V^{*},E^{*})$ of $(v_l,v_m)$ in $G-e(a,b)$ which is shorter than any element of $MSPG_{lm}$. Then, $MSPG_{lm}$ is a path group of $(v_l,v_m)$ in $G-e(a,b)$ but not a shortest path group. Adding a new edge between node pair $(a,b)$ back to $G-e(a,b)$ to obtain $G$, the path $G_{P,lm}^{*}(V^{*},E^{*})$ still exists in $G$. This contradicts to the original assumption that $MSPG_{lm}$ is the maximum shortest path of $(v_l,v_m)$ in $G$. Therefore, $\Delta d(v_l, v_m) = 0$ in this case.

(ii) The edge $e(a,b)$ belongs to some elements of $MSPG_{lm}$, but there exists at least one element of $MSPG_{lm}$ which does not contain $e(a,b)$. Thus, there exist $G_{P,lm}^1(V_1, E_1) \in MSPG_{lm}$ and $G_{P,lm}^2(V_2, E_2) \in MSPG_{lm}$ such that $e(a,b) \in E_1$ and $e(a,b) \notin E_2$. Then, the path subgraph $G_{P,lm}^1(V_1, E_1)$ is not in $G-e(a,b)$ whereas $G_{P,lm}^2(V_2, E_2)$ is still in $G-e(a,b)$. Similar to the above (i), there does not exist a shorter path of $(v_l,v_m)$ in $G-e(a,b)$. Thus, $\Delta d(v_l, v_m) = 0$.

(iii) All elements of $MSPG_{lm}$ contain the edge $e(a,b)$. If $e(a,b)$ is a bridge, then $d_{G-e(a,b)}(v_l,v_m) = \infty$. Otherwise, let $MPG_{lm}$ be the maximum path group of $(v_l,v_m)$ in $G$, then $\forall~G_{P,lm}(V_P,E_P) \in MPG_{lm} - MSPG_{lm}$, we have $|E_P| \geq d_G(v_l,v_m) + 1$. The shortest path of $(v_l,v_m)$ in $G-e(a,b)$ belongs to $MPG_{lm} - MSPG_{lm}$. It follows that $\Delta d(v_l, v_m) \geq 1$.                        
\end{proof}

\begin{remark} 
According to Lemma \ref{lem:edge_removal}, when an edge $e\in E$ is removed from a connected simple graph $G(V,E)$, the geodesic distances remain unchanged for some node pairs while will increase for some other node pairs. Therefore, removing an edge from $G$ will lead to an increase in the average distance $\bar{D}$ of $G$.
\label{rmk:4}
\end{remark}

\begin{remark}
For any edge $e(a,b)\in E$, if the geodesic distance $d_G(v_l,v_m)$ of $(v_l,v_m)$ increases with the removal of $e$ from $G$, then $\Delta d(v_l, v_m) \leq d_{G-e}(a,b) - 1$.
If $G-e(a,b)$ is disconnected, $\Delta d(v_l, v_m) = d_{G-e}(a,b) - 1 =+\infty$; otherwise, $\Delta d(v_l, v_m) \leq d_{G-e}(a,b) - 1 < +\infty$ since both $G_{P,la}+G_{P,ab}+G_{P,bm}$ and $G_{P,lb}+G_{P,ba}+G_{P,am}$ are not necessarily the shortest path. For example, $d_{G-e}(v_l,v_m) = d_G(v_l,v_m) + (d_{G-e}(a,b) - 1)$ in case (i) of Fig.~\ref{fig:compre_evo}a) and $d_{G-e}(v_l,v_m) < d_G(v_l,v_m) + (d_{G-e}(a,b) - 1)$ in case (ii) of Fig.~\ref{fig:compre_evo}a).  
\label{rmk:5}
\end{remark}

\begin{lemma}[Edge addition] 
\label{lem:edge_addition}
Let $G(V,E)$ denote a connected simple graph. Assume there does not exist an edge between a pair of existing nodes $(a,b)$ in $G$.
For any other node pair $(v_l, v_m)$, let $\Delta d(v_l, v_m)$ denote the distance variation of the node pair $(v_l, v_m)$ after a new non-duplicate edge $e(a,b)$ is added to $G$. Then, 
\begin{align}
&\Delta d(v_l, v_m) = 0, ~\text{or}~=1 - d_G(a,b), \text{~or~} = \min\{d_{G}(v_l,a)\nonumber\\
&~~+d_{G}(v_m,b)+1, d_{G}(v_l,b)+d_{G}(v_m,a)+1\} \nonumber\\
&~~- d_G(v_l,v_m), \text{~when~} \min\{d_{G}(v_l,a)+d_{G}(v_m,b)+1, \nonumber\\
&~~d_{G}(v_l,b) +d_{G}(v_m,a)+1\} < d_G(v_l,v_m).  
\label{eqn:Deltad_add}
\end{align}
\end{lemma}

\begin{proof}  
Since node pair $(a,b)$ is not directly connected, $d_G(a,b) \geq 2$. For any node pair $(v_l,v_m)$ in $G(V,E)$, let $MSPG_{lm}$ and $MPG_{lm}$ denote the maximum shortest path group and the maximum path group of $(v_l,v_m)$ in $G$, respectively. We have the following two cases:

(i) Both $a$ and $b$ belong to at least one shortest path of $(v_l,v_m)$. For a path subgraph $G_{P,lm}^{*}(V^{*},E^{*}) \in MSPG_{lm}$ with $\{a,b\} \subset V^{*}$, the path of $(a,b)$ in $G_{P,lm}^{*}(V^{*},E^{*})$ is a shortest path of $(a,b)$ (denote as $G_{P,ab}(V_{ab}, E_{ab})$). Thus, we have $d_G(a,b) = |E_{ab}| \geq 2$ and $E_{ab} \subseteq E^{*}$. 

When a new edge $e(a,b)$ is added to $G$, $G_{P,lm}^{*}(V^{*},E^{*})$ can also be added with the new edge $e(a,b)$ to become $G_{P,lm}^{*}(V^{*},E^{*})+e(a,b)$. The graph $G_{P,lm}^{*}(V^{*},E^{*})+e(a,b)$ is not a path since the conditions (ii) and (iii) in Lemma~\ref{lem:path_subgraph} do not hold. Denote $G_{P,lm}^{'}(V^{'}, E^{'})$ as the  shortest path of $(v_l,v_m)$ in $G+e(a,b)$. Then, we have $V^{'} = V^{*} - V_{ab} + \{a,b\}$ and $E^{'} = E^{*} - E_{ab} + \{e(a,b)\}$. It follows that
\begin{align}
d_{G+e}(v_l, v_m) &= |E^{'}| = |E^{*}| - |E_{ab}| +1 \nonumber\\
&= d_G(v_l,v_m) - d_G(a,b) + 1.
\label{eqn:lem3proof}
\end{align}

(ii) Either or both of nodes $a$ and $b$ do not belong to the shortest path of $(v_l,v_m)$ in $G$. This means that $\{a,b\} \not\subset V^{*}$ for any $G_{P,lm}^{*}(V^{*},E^{*}) \in MSPG_{lm}$. Let $G_{P,1}(V_1,E_1)$, $G_{P,2}(V_2,E_2)$, $G_{P,3}(V_3,E_3)$ and $G_{P,4}(V_4,E_4)$ be shortest paths of $(v_l,a)$, $(v_m,b)$, $(v_l,b)$ and $(v_m,a)$ in $G$, respectively. 

If $\min\{d_G(v_l,a)+d_G(v_m,b)+1, d_G(v_l,b)+d_G(v_m,a)+1\} < d_G(v_l,v_m)$, without loss of generality, assume $d_G(v_l,a)+d_G(v_m,b)+1 < d_G(v_l,b)+d_G(v_m,a)+1$, then a new edge $e(a,b)$ can connect the paths $G_{P,1}(V_1,E_1)$ and $G_{P,2}(V_2,E_2)$ to obtain a new path after adding $e(a,b)$ in $G$. Denote the new path as $G_{P}^{'}(V^{'},E^{'})$. It is seen that $G_{P}^{'}(V^{'},E^{'})$ is a shortest path of $(v_l,v_m)$ in $G+e(a,b)$ because $G_{P,1}(V_1,E_1)$ and $G_{P,2}(V_2,E_2)$ are shortest paths of $(v_l,a)$ and $(v_m,b)$, respectively. If $\min\{d_G(v_l,a)+d_G(v_m,b)+1, d_G(v_l,b)+d_G(v_m,a)+1\} > d_G(v_l,v_m)$, then $MSPG_{lm}$ is still the shortest path group of $(v_l,v_m)$ in $G+e(a,b)$, implying that $d_{G+e}(v_l,v_m) = d_G(v_l,v_m)$.   
\end{proof}

\begin{remark} According to Lemma~\ref{lem:edge_addition}, after a new non-duplicate edge $e$ is added to a simple graph $G(V,E)$, the geodesic distances remain the same for some node pairs but will decrease for other node pairs. This highlights the fact that adding an edge to $G$ will result in a reduction in the average distance $\bar{D}$ of $G$.
\label{rmk:6}
\end{remark}

\subsection{Bounds of network average-distance variations}

From the above analysis, the average distance $\bar{D}$ of a network $G(V,E)$ will increase when an edge $e\in E$ is removed from $G$. 
Conversely, $\bar{D}$ will decrease following the addition of a new non-duplicate edge $e$ between two existing nodes in $G$.  But how and where to remove or add an edge? This requires a rigorous theory to answer. Then, an effective compression evolution can be built on this theory. In the following, we will present our theory and conceptual design of compressional evolution for networks.
 
\begin{definition}[Edge betweenness]  
\label{def:4}
Let $G(V,E)$ be a connected simple graph.
For any edge $e(a, b) \in E$, let $MSPG_{lm}$ and $SPG_{lm,e}$ denote the maximum shortest path group of $(v_l,v_m)$ and shortest path group with its all path graphs containing $e$, respectively. Then, the edge betweenness $B_G(e)$ of $e$ is defined as
\begin{equation}
B_G(e) = \hspace{-1ex}\sum_{\substack{l,m; l\neq m,\{v_l,v_m\} \neq \{a,b\} }}\hspace{-2ex} {|SPG_{lm,e}|}/{|MSPG_{lm}|}. 
\end{equation}
\end{definition}

\noindent\textit{Edge removal upper-bound of $\bar{D}$ increase:}

\begin{theorem}[Upper bound] Let $\Delta \bar{D}$ be the variation of $\bar{D}$ after the removal of edge $e(a,b)$ from $G(V,E)$. 
Then, 
\begin{equation}
{\Delta \bar{D}} \leq \frac{(B_G(e)+2)(d_{G-e}(a,b)-1)}{|V|(|V|-1)} \triangleq \overline{\Delta D_e^-}. 
\label{eqn:DeltaD-e}
\end{equation}
$\overline{\Delta D_e^-}$ is an upper bound of $\bar{D}$ increment resulting from the removal of $e$.
\label{thm:upper_bound}
\end{theorem}

\begin{proof} 
Assume $e(a,b) \in E$ is an edge to be removed from $G(V,E)$, and after $e$ is removed the geodesic distance of the node pair $(a,b)$ is $d_{G-e}(a, b)$. For any $l,m \in \{1,2,\cdots,|V|\}$, $l\neq m$ and $\{v_l,v_m\} \neq \{a,b\}$, let $MSPG_{lm}$ and $SPG_{lm,e}$ be the maximum shortest path group of $(v_l,v_m)$ and the shortest path group in which each path passes through the edge $e$ in $G$, respectively. According to the proof of Lemma~\ref{lem:edge_removal}, we have $|SPG_{lm,e}|=0$, or $0<|SPG_{lm,e}|<|MSPG_{lm}|$, or $|SPG_{lm,e}|=|MSPG_{lm}|$. $\bar{D}$ can increase after $e$ is removed from $G$ if and only if  $|SPG_{lm,e}|=|MSPG_{lm}|$. 
According to Remark~\ref{rmk:5}, we have  
\begin{align}
&\Delta \bar{D} = \frac{1}{|V|(|V|-1)} \sum_{\substack{ l,m; l\neq m }}(d_{G-e}(v_l, v_m) - d_G(v_l, v_m)) \notag\\
&\leq \frac{(d_{G-e}(a,b)-1)}{|V|(|V|-1)}\left(2+\hspace{-3ex}\sum_{\substack{ l,m;l\neq m,\\ \{v_l,v_m\} \neq \{a,b\} }} \hspace{-3ex}\frac{|SPG_{lm,e}|}{|MSPG_{lm}|}\right), \nonumber
\end{align}
giving inequality (\ref{eqn:DeltaD-e}). This complets the proof.
\end{proof}

\noindent\textit{Edge addition lower-bound of $\bar{D}$ decrease:}

\begin{theorem}[Lower bound] Let $\Delta \bar{D}$ be the $\bar{D}$ variation when a new non-duplicate edge $e(a,b)$ is added to $G(V,E)$. Also, after the addition of $e$, the distance variation of node pair $(v_l, v_m)$ is denoted as $\Delta d(v_l,v_m)$. Then, 
\begin{align}
& \Delta \bar{D} \leq \frac{(B_{G+e}(e)+2)\Delta d_{max,e}}{|V|(|V|-1)} \triangleq -\underline{\Delta D_e^+}, 
\label{eqn:DeltaD+e}\\
& \Delta d_{max,e} =\hspace{-2ex} \max_{\shortstack{$v_l \neq v_m,$\\ $v_l, v_m \in V$}} \hspace{-2ex}\{\Delta d(v_l,v_m) \mid \Delta d(v_l,v_m) \neq 0\}. 
\label{eqn:Deltad}
\end{align}
$\Delta \bar{D}<0$ and the decrement of $\bar{D}$ equals to $-\Delta \bar{D}$. 
Thus, $\underline{\Delta D_e^{+}}>0$ is a lower bound of $\bar{D}$ decrement due to the addition of the edge $e$.
\label{thm:lower_bound}
\end{theorem}

\begin{proof} 
Assume $a\in V$ and $b\in V$ are not directly connected and $e(a, b)$ is a new edge to be added into $G$.  Let $MSPG_{lm}$ and $SPG_{lm,e}$ denote the maximum shortest path group of $(v_l,v_m)$ and shortest path group which paths passes through the edge $e$ in $G + e$, respectively. For each node pair $(v_l,v_m)$, according to the proof of  Lemma~\ref{lem:edge_addition}, case (i) will cause the distance of the node pair to decrease by $d_G(a,b)-1$,  and case (ii) may lead to a reduction in the distance by $d_G(v_l,v_m)-\min\{d_G(v_l,a)+d_G(v_m,b)+1, d_G(v_l,b)+d_G(v_m,a)+1\}$ when $\min\{d_G(v_l,a)+d_G(v_m,b)+1, d_G(v_l,b)+d_G(v_m,a)+1\} < d_G(v_l,v_m)$. In both cases, we have $MSPG_{lm} = SPG_{lm,e}$ in $G+e$ if the distance between $(v_l,v_m)$ is reduced. Hence, the edge betweenness $B_{G+e}(e)+2$ is actually the number of node pairs that cause the distance to decrease after the addition of $e$ into $G$. Thus, 
\begin{eqnarray}
\Delta \bar{D} &=& \frac{1}{|V|(|V|-1)} \sum_{\substack{ l,m; l\neq m }}(d_{G+e}(v_l,v_m) - d_{G}(v_l,v_m)) \nonumber\\
&=& \frac{1}{|V|(|V|-1)} \sum_{\substack{ l,m; l\neq m }} (\Delta d(v_l,v_m))    \nonumber 
\end{eqnarray}
yielding inequality (\ref{eqn:DeltaD+e}). This completes the proof.  
\end{proof}

\noindent\textit{Conceptual rewiring design for topological compression:}

A rewiring step for a network consists of two sequential operations: edge removal and edge addition~\cite{LYA2023}. 
When an edge is removed from, or added to, a network, the resulting increase or decrease in $\bar{D}$ can be controlled by the upper bound $\overline{\Delta D_{e}^{-}}$ or lower bound $\underline{\Delta D_{e}^{+}}$ according to Theorem \ref{thm:upper_bound} or Theorem \ref{thm:lower_bound}. Therefore, effective rewiring to compress network topology should satisfy the following conditions:  
\begin{enumerate}
\renewcommand{\labelenumi}{(\roman{enumi})}
\item Edge removal: The edge with a smaller $\overline{\Delta D_{e}^{-}}$, when removed, should be selected for removal. This will cause a smaller increase in $\bar{D}$.

\item Edge addition: A new non-duplicate edge that gives a larger $\underline{\Delta D_{e}^{+}}$, when added, should be chosen to add for a larger reduction in $\bar{D}$.
\end{enumerate}

We provide the following network compression evolution.

\subsection{Mathematical formalism of effective compression evolution for network topology}

For a connected simple graph $G(V,E)$ representing a network, let $P_{rew}>1/|E|$ be the graph evolution fraction.  Then, the number of graph evolution steps is $N_{max}=\lceil P_{rew}|E|\rceil > 1$. Let $\mathcal{T}_{cut}(v_{1},v_{2})$ and $\mathcal{T}_{add}(v_{1},v_{2})$ be the removal and addition operators of edge $e(v_{1},v_{2})$, respectively. 
From the theory established above, we developed the following compression evolution for a network:\\ 
\begingroup
\allowdisplaybreaks
\begin{align}
&\underbrace{G(V,E)}_{\substack{ \text{ Initial network }}}\nonumber\\
&~\underbrace{%
\xrightarrow{\mathcal T_{cut}(a,b)}G_{1}^{cut}(V,E_{1}^{cut})%
\xrightarrow{\mathcal T_{add}(a,a_1)}G_{1}^{add}(V,E_{1}^{add})}_{\substack{ 
\text{step 1}}}\cdots \nonumber\\
&~\underbrace{\xrightarrow{\mathcal T_{cut}} G_{i}^{cut}(V,E_{i}^{cut})\xrightarrow{\mathcal T_{add}}G_{i}^{add}(V,E_{i}^{add})}_{\substack{ \text{step $i,i\geq 2$}}}\cdots \nonumber \\
&~\underbrace{\xrightarrow{\mathcal
T_{cut}}%
G_{N_{max}}^{cut}(V,E_{N_{max}}^{cut})\xrightarrow{\mathcal
T_{add}}G_{N_{max}}^{add}(V,E_{N_{max}}^{add})} 
_{\substack{\text{ step $N_{max}$}}}
\label{eqn:net_evo}
\end{align}%
where for
\begin{align*}
&    \text{Step~}i \geq 2:
    \left\{%
    \begin{array}{@{}r@{~}l@{}}
        \mathcal T_{cut}: &\mathcal T_{cut}(a_{2(i-2)+1},a_{2(i-1)}),\\
        \mathcal T_{add}: &\mathcal
T_{add}(a_{2(i-1)},a_{2(i-1)+1}),
    \end{array}
    \right.%
    %
    \\
&    \text{Step~}N_{max}:
    \left\{\begin{array}{@{}r@{~}l@{}}
        \mathcal T_{cut}: &\mathcal
T_{cut}(a_{2(N_{max}-2)+1},a_{2(N_{max}-1)}),\\
        \mathcal T_{add}: &\mathcal
T_{add}(a_{2(N_{max}-1)},b).
    \end{array}
    \right.
\end{align*}
\endgroup

For the evolution process shown in Eq. (\ref{eqn:net_evo}), we set $e(a,b)=\mathop{\arg \min}_{e \in E} B_G(e)$ as the initial edge to remove. Then, we find node pair $(a,a_1)$ to add a new non-duplicate edge $e(a,a_1)$ to $G$ such that $a_1 = \mathop{\arg\max}_{v \in V_{a_1}^A}$ $B_{G_1^{add}}(e(a,v))d_{G_1^{cut}}(a,v)$, where $d_{G_1^{cut}}(a,\cdot)$ and $B_{G_1^{add}}(e(a,\cdot))$ are the distance of the node pair $(a,\cdot)$ in $G_1^{cut}$ and the edge betweenness of edge $e(a,\cdot)$ in $G_1^{add}$, respectively. 
\begin{align}
V_{a_1}^A = \left\{ 
\begin{array}{ll}
V-\bar{V}_{a}, &G_1^{cut}(V, E_1^{cut})~\text{connected};\\[4pt] 
V_{1,1}-\{b\}, & 
G_1^{cut}(V, E_1^{cut})=G_{1,1}^{cut}(V_{1,1},
E_{1,1}^{cut})+\\[4pt]
\multicolumn{2}{l}{\hspace{\mytmpindentfour}G_{1,2}^{cut}(V_{1,2}, E_{1,2}^{cut})~\text{disconnected}, a \in
V_{1,2};}\\ [4pt]
V_{1,2}-\{b\}, & 
G_1^{cut}(V, E_1^{cut})=G_{1,1}^{cut}(V_{1,1},
E_{1,1}^{cut})+\\[4pt]
\multicolumn{2}{l}{\hspace{\mytmpindentfour}G_{1,2}^{cut}(V_{1,2}, E_{1,2}^{cut}) ~\text{disconnected}, a \in
V_{1,1},}%
\end{array}%
\right.
\end{align}%
where $\bar{V}_{a}$ represents the set of node $a$ and its neighboring nodes in $G(V,E)$. 
In the initial step, we set the smaller degree node of initial edge endpoint as $a$, ensuring network connectivity after the first evolution step.
Let $a_0=a$ and for $2\leq i\leq N_{max}$, set $a_{2(i-1)}=\mathop{\arg \min}_{v \in V_{a_{2(i-1)}}^A} [B_{G_{i-1}^{add}}(e(a_{2(i-2)+1}, v))d_{G_i^{cut}}(a_{2(i-2)+1}, v)]$, where $V_{a_{2(i-1)}}^A = V_{a_{2(i-2)+1}} -\{a_{2(i-2)}\}$, $V_{a_{2(i-2)+1}}$ is the set of neighbor nodes of node $a_{2(i-2)+1}$ in $G_{i-1}^{add}(V,E_{i-1}^{add})$, and $d_{G_i^{cut}}(a_{2(i-2)+1}, v)$ and $B_{G_{i-1}^{add}}(e(a_{2(i-2)+1}, v))$ are the distance of node pair $(a_{2(i-2)+1}, v)$ in $G_i^{cut}$ and the edge betweenness of the edge $e(a_{2(i-2)+1}, v)$ in $G_{i-1}^{add}$, respectively. Then we remove the edge $e(a_{2(i-2)+1}, a_{2(i-1)})$ in $G_{i-1}^{add}$ to obtain $G_{i}^{cut}$. 
For $2\leq i\leq N_{max}-1$, set $a_{2(i-1)+1} = \mathop{\arg \max}_{v \in V_{a_{2(i-1)+1}}^A} [B_{G_i^{add}}(e(a_{2(i-1)}, v))d_{G_i^{cut}}(a_{2(i-1)}, v)]$, where 
\begin{align}
&V_{a_{2(i-1)+1}}^A = \nonumber\\[4pt]
&\left\{ 
\begin{array}{ll}
V-\bar{V}_{a_{2(i-1)}}, ~G_i^{cut}(V, E_i^{cut})~\text{connected}; &\\[4pt]
V_{i,1}-\{a_{2(i-2)+1}\}, &\\[4pt]
\hspace{\mytmpindentfour}G_i^{cut}(V,E_i^{cut})=G_{i,1}^{cut}(V_{i,1}, E_{i,1}^{cut})+&\\[4pt]
\hspace{\mytmpindentfour}G_{i,2}^{cut}(V_{i,2},
E_{i,2}^{cut})~\text{disconnected}, a_{2(i-1)}\in V_{i,2}; &\\[4pt]
V_{i,2}-\{a_{2(i-2)+1}\}, &\\[4pt]
\hspace{\mytmpindentfour}G_i^{cut}(V,
E_i^{cut})=G_{i,1}^{cut}(V_{i,1}, E_{i,1}^{cut})+&\\[4pt]
\hspace{\mytmpindentfour}G_{i,2}^{cut}(V_{i,2},
E_{i,2}^{cut})~\text{disconnected}, a_{2(i-1)}\in V_{i,1}.&%
\end{array}%
\right.
\label{eqn:add_edge_i_node_set}
\end{align}%
In Eq. (\ref{eqn:add_edge_i_node_set}), $\bar{V}_{a_{2(i-1)}}$ represents the set of node $%
a_{2(i-1)}$ and its neighboring
nodes in $G_{i-1}^{add}(V,E_{i-1}^{add})$, $G_{i,1}^{cut}(V_{i,1},E_{i,1}^{cut})$ and $%
G_{i,2}^{cut}(V_{i,2},E_{i,2}^{cut})$ are two connected components of $%
G_{i}^{cut}(V,E_{i}^{cut})$ if $G_{i}^{cut}$ disconnected. Then we add a non-duplicate edge $e(a_{2(i-1)}, a_{2(i-1)+1})$ in $G_i^{cut}$.

\begin{figure*}[htb!] 
\centering 
\includegraphics[width=0.9\linewidth,trim=76 420 72 50,clip]{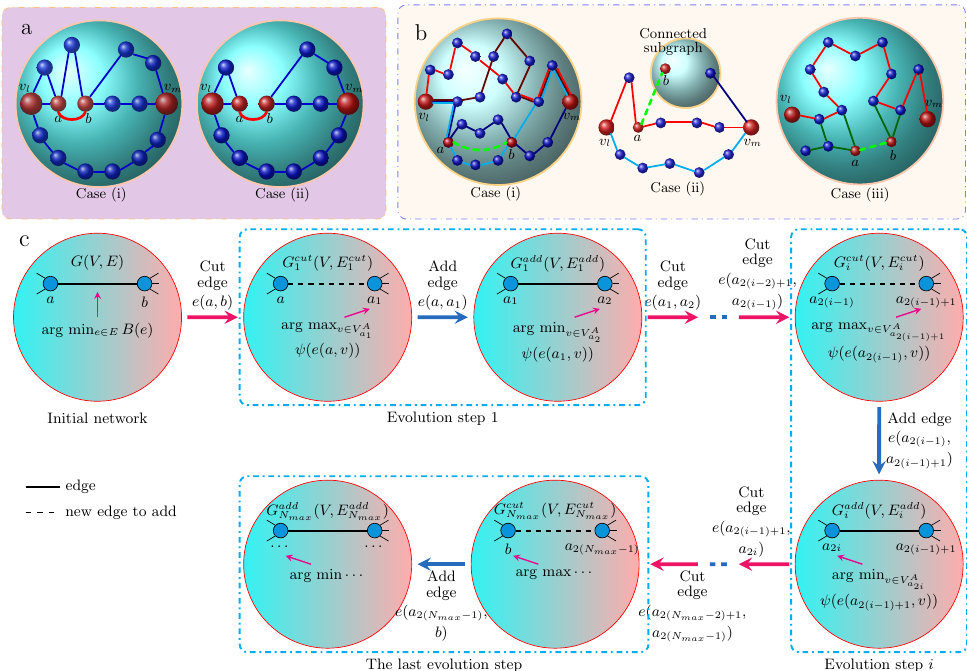}
\caption{%
\textbf{Schematic of network compression evolution.}
\textbf{a}, Typical cases of edge cutting in $G(V,E)$. In each of cases (i) and (ii), the end node of edge $e(a,b)$ to be cut belongs to the shortest path of $(v_l,v_m)$. For case (i), the increment of distance of node pair $(v_l,v_m)$ equals to $d_{G-e}(a,b)-1$. For case (ii), it is smaller than $d_{G-e}(a,b)-1$.
\textbf{b}, Typical cases for adding a new edge in $G(V,E)$. Case (i): there are four shortest paths (in different colors to identify) between $(v_l,v_m)$, and both two endpoints of the new edge belong to the shortest path of $(v_l,v_m)$. Case (ii): there are two shortest paths (in different colors) between $(v_l,v_m)$, and one end node of the new edge belong to the shortest path of $(v_l,v_m)$. Case (iii): both nodes of the new edge are not in the shortest path of $(v_l,v_m)$.
\textbf{c}, Network compression evolution. 
It is a sequence of iterative rewiring operations including cutting an edge followed by adding a new edge. 
}
\label{fig:compre_evo}
\end{figure*}

In the final step of the graph evolution process, if $G_{N_{max}}^{cut}(V,E_{N_{max}}^{cut})$ is disconnected, then $G_{N_{max}}^{cut}$ contain only two connected components, i.e., 
\begin{align}
G_{N_{max}}^{cut}(V,E_{N_{max}}^{cut}) &= G_{N_{max},1}^{cut}(V_{N_{max},1},E_{N_{N_{max},1}}^{cut}) \nonumber\\
&+ G_{N_{max},2}^{cut}(V_{N_{max},2},E_{N_{N_{max},2}}^{cut}). 
\end{align}
If $a_{2(N_{max}-1)}$ and $b$ belong to the same connected component, without loss of generality, 
let $\{a_{2(N_{max}-1)},b\}\subset V_{N_{max},2}$, then the connectivity of evolution network can hold though the following additional evolution steps:
\begin{align*}
& \underbrace{G_{N_{max}}^{cut}(V,E_{N_{max}}^{cut})%
\xrightarrow%
{\mathcal
T_{add}}%
G_{N_{max}}^{add}(V,E_{N_{max}}^{add})}_{\substack{ \text{ step $N_{max}$ }}}
\nonumber\\
& \underbrace{\xrightarrow{\mathcal T_{cut}}G_{N_{max}+1}^{cut}(V,E_{N_{max}+1}^{cut})%
\xrightarrow{\mathcal T_{add}}%
G_{N_{max}+1}^{add}(V,E_{N_{max}+1}^{add})}_{\substack{ \text{ step $%
N_{max}+1$ }}}
\end{align*}%
where for 
\begin{align*}
&    \text{Step~}N_{max}:
    \left\{%
    \begin{array}{@{}r@{~}l@{}}
        \mathcal T_{add}: &\mathcal
T_{add}(a_{2(N_{max}-1)},a_{2(N_{max}-1)+1}),
    \end{array}
    \right.%
    %
    \\
&    \text{Step~}N_{max}+1:
    \left\{\begin{array}{@{}r@{~}l@{}}
        \mathcal T_{cut}: &\mathcal T_{cut}(a_{2(N_{max}-1)+1}, a_{2N_{max}}),\\[4pt]
        \mathcal T_{add}: &\mathcal
T_{add}(a_{2N_{max}}, b).
    \end{array}
    \right.
\end{align*}
In above steps, set 
\begin{flalign*}
&a_{2(N_{max}-1)+1} =
\mathop{\arg \max}_{v \in V_{a_{2(N_{max}-1)+1}}^A} 
B_{G_{N_{max}}^{add}}(e(a_{2(N_{max}-1)},v))\\ 
&\hspace{3em}\times d_{G_{N_{max}^{cut}}}(a_{2(N_{max}-1)}, v),\\
&a_{2N_{max}} =
\mathop{\arg \min}_{v \in V_{a_{2N_{max}}}^A} 
B_{G_{N_{max}}^{add}}(e(a_{2(N_{max}-1)+1},v))\\
&\hspace{3em}\times d_{G_{N_{max}+1}^{cut}}(a_{2(N_{max}-1)+1},v),
\end{flalign*}
where 
\begin{align*}
\begin{cases}
V_{a_{2(N_{max}-1)+1}}^A = V_{N_{max},1}-\{a_{2(N_{max}-2)+1}\},\\
V_{a_{2N_{max}}}^A = V_{a_{2(N_{max}-1)+1}}-\{a_{2(N_{max}-1)}\},
\end{cases}
\end{align*}
$V_{a_{2(N_{max}-1)+1}}$ is the set of the neighbor nodes of node $a_{2(N_{max}-1)+1}$ in $G_{N_{max}}^{add}(V,E_{N_{max}}^{add})$. 

The overall compression evolution process of network topology are illustrated in Fig.~\ref{fig:compre_evo}c). We refer to the product of the distance between node pairs and edge betweenness as the local compression modulus, denoted by $\psi$.

\subsection{Incorporating node constraints into network compression evolution}

The compression evolution described above is applicable to generic complex networks represented by connected simple graphs. In each evolution step, it removes an existing edge and adds a new edge merely based on maximal network compression. However, there exist many real complex networks, such as scale-free neighborhood area networks in smart grid~\cite{DLT2018}, in which node pairs cannot be arbitrarily connected to each other, e.g., due to a limited radio range of a wireless node. Considering node constraints, we have extended the the above general compression evolution to complex networks with node constraints.

Consider a connected simple graph $G(V,E)$. 
Assume that each vertex has some constraints or requirements characterized by a free variable or vector $\delta_i, i\in \{1, 2,\cdots, |V|\}$. Thus, the vertex set $V$ is represented by $V = \{v_1(\delta_1), v_2(\delta_2), \cdots, v_{|V|}(\delta_{|V|})\}$. 

\begin{definition}[Node constraint condition] Let $G(V,E)$ be a simple graph, where $E$ is the set of edges and $V = \{v_1(\delta_1), v_2(\delta_2), \cdots, v_{|V|}(\delta_{|V|})\}$ is the set of vertices with constraints characterized by a freedom value (or vector) $\delta_i, i=\{1, 2, \cdots, |V|\}$. For each node $v_i\in V$, define a constraint function $f_{v_i(\delta_i)}(v(\delta))$, which quantify the relationship between $v_i$ with other node $v$ by their freedom value (or vector). Then, for $i\in\{1,2,\cdots,|V|\}$, a constraint condition $f_{v_i(\delta_i)}(v(\delta))<c_i$ is defined for node $v_i$, where $c_i$ is a constant. 
\label{def:5}
\end{definition}%

For example, if $\delta_i$ of node $v_i$ is a two-dimension vector of the node coordinates, then the constraint for connectivity to node $v_j$ can be defined as
\begin{equation*}
    f_{v_i(\delta_i)}(v_j (\delta_j)) = \| \delta_i - \delta_j \|_2 
    \leq c_i,~v_i\neq v_j, v_i,v_j\in V. 
\end{equation*}

\begin{remark} 
For a vertex $v_i\in V$, selecting a vertex subset $V_i\subseteq V$ by applying the node constraint condition yields
\begin{equation*}
    V_i = \{v_j \mid v_j \in V, f_{v_i(\delta_i)}(v_j(\delta_j)) < c_i, j=1,2,\cdots,|V| \}.
    \label{eqn:Vi}
\end{equation*}
Then, search in $V_i$ rather than the entire $V$ for a node to add a new non-duplicate edge.
\label{rmk:7}
\end{remark}

\section{Algorithm design and complexity analysis}

\subsection{Algorithm for implementing network compression evolution}

For a given complex network $G(V,E)$, a top-level algorithm (\textbf{Algorithm 1}) for our compression evolution is as follows: 
\begin{enumerate}[leftmargin=\widthof{\hspace{\parindent}~Step 12.},labelwidth=\widthof{\hspace{\parindent}~Step 12.}]
\item[Step 1.] Calculate the total number of graph evolution iterations $N_{max}>1$ according to a given fraction $P_{rew}>1/|E|$. 

\item[Step 2.] Compute the edge betweenness of the network, select the initial edge $e=e(a,b)$ with the smallest edge betweenness as the edge to be removed to start the network evolution process. The node with smaller degree between $a$ and $b$ be set as the second end node of this edge. Assume that the selected node is $a$.

\item[Step 3.] If edge has not been selected for removal, set the second end node of $e$ as the first end node of the edge to be removed. All edges connecting to the end node are hypothetical edges for possible removal except for the edge added in previous step. Search the hypothetical edges with the smallest local compression modulus $\psi$ among all hypothetical edges to determine the edge $e$ to be removed. 

\item[Step 4.] Remove the edge $e$. 

\item[Step 5.] Set the second end node of the removal edge $e$ as the first end node of the edge to be added.

\item[Step 6.] Determine where to search a candidate node as the second end node of the new edge to be added. A candidate node should be distinct from the first end node of removal edge in previous step, the first end node of new edge and any of its adjacent nodes, but must be in the connected component of another end node of removal edge in previous step to ensure network connectivity. If the node constrained is needed for network evolution, determine a subset of the space determined above for search a candidate node by using \textbf{Algorithm 2} presented below.

\item[Step 7.] Search all nodes in the search space determined in Step 6. Select a node from the search space as the second end node of the new edge $e$ to be added such that the local compression modulus $\psi$ of the new edge on the entire graph is the largest among all hypothetical new edges. 

\item[Step 8.] Add the new edge $e$. 

\item[Step 9.] If the evolution has not reached $N_{max}-1$, repeat Steps 3 through 8.

\item[Step 10.] In the final evolution step ( $N_{max}$), in a similar manner to Steps 3 and 4, determine an edge $e$ to remove, and then remove it. 

\item[Step 11.] If $b$ is an end node of $e$, in a similar manner to Steps 5 through 7, determine a new edge $e$ to add in network. Then, similar to Steps 3 and 4, determine an edge to remove, and remove it. 

\item[Step 12.] Add the last edge to $b$. Evolution ends. 
\end{enumerate}

\noindent\textbf{Algorithm 2} to determine search space with node constraints:
\begin{enumerate}[leftmargin=\widthof{\hspace{\parindent}~Step 6.3.},labelwidth=\widthof{Step 6.3.~}]
\item[Step 6.1.] Assign the node free values (or vectors) as global parameters for network nodes to construct the node constraint conditions. The global parameters do not reassign once it be assigned in sequence evolution steps. 

\item[Step 6.2.] Construct the node constraint functions and constraint conditions.

\item[Step 6.3.] Select a node subset of search space determined in Step 6 through the node constraint condition as the new search space.
\end{enumerate}

\subsection{Time complexity of the algorithm}

Our compression evolution algorithm requires calculating the edge betweenness and the distances between node pairs in each step to determine which edge need to remove and which one to add. The fastest known algorithms for this task are the Brandes fast algorithm and the breadth-first search (BFS) algorithm, particularly suitable for unweighted networks. For a given unweighted graph $G(V,E)$, the computational complexities of calculating the distance between a node pair $(a,b)$ and the edge betweenness of $e \in E$ are $O(|E|)$ and $O(|V|\cdot |E|)$, respectively~\cite{BRAN2001}.


Consider $G(V,E)$ as a connected simple graph with $k(v)$ denoting the degree of node $v$ in $G$. In the first step of the graph evolution, we calculate the edge betweenness of all edges in $G(V,E)$ and select the minimal one as the initial edge $e(a,b)$ to remove. This sub-step requires $O(|V| |E|^2)$ times. Subsequently, computation is required for the distance and betweenness of $|V|-k(a)-1$ node pairs and edges, respectively. For a source node $a \in V$, BFS has the complexity $O(|E|)$ to obtain the distances to all other nodes. Computing these edge betweenness values requires $O((|V|-k(a)-1)|V||E|)$ times. 
For step $i, 1<i<N_{max}$,  $O(|E|+k(a_{2(i-2)+1}))$ times are needed to determine the edge $e(a_{2(i-2)+1}, a_{2(i-1)})$ to remove, and $O(|E|+(|V|-k(a_{2(i-1)})-1)|V||E|)$ times to determine the edge $e(a_{2(i-1)}, a_{2(i-1)+1})$ to add. In the final evolution step, only $O(|E|+k(a_{2(N_{max}-2)+1}))$ times are needed to determine the edge $e(a_{2(N_{max}-2)+1}, a_{2(N_{max}-1)})$ to remove. Finally, the edge is directly added to $b$. Therefore, the time complexity of our compression evolution algorithm is $O(|V|^2\cdot |E|^2)$.

\section{Experimental design and results}

\subsection{Model networks for validation and demonstration}

\textbf{Barab{\'a}si-Albert network.} 
The Barab{\'a}si-Albert (BA) model is a well-known network growth model used to generate a scale-free network by following the preferential attachment rule \cite{BA1999}. It can simulate real complex networks with a power-law degree distribution. Starting with a small initial connected network, new nodes are continuously added, preferentially connecting to existing nodes with a higher degree. The model has two parameters: the network size ($|V|$) and the number of nodes ($m$) to which each new node will attach. Specifically, a newly added node at a given time step will be connected to $m$ existing nodes based on a connection probability $P_i = \frac{k_i}{\sum_j k_j}$. This process is iterated until the predetermined network size is reached.

\textbf{Watts-Strogatz network.}  
The WS model is another well-known model used to simulate real complex networks with small-world characteristics \cite{WS1998}. The model has three parameters: the network size (${|V|}$), the initial node degree ($k$), and the edge rewiring probability ($p$). It starts with a $k$-nearest neighbors network with ${|V|}$ nodes, and then each edge is rewired with probability $p$ which is set as 0.5 in all corresponding experiments. 

\textbf{Erd{\" o}s-R{\'e}nyi network.}
The Erd{\" o}s-R{\'e}nyi (ER) network is a classical random network. Given an initial node set $V$, edges are randomly connected between any pair of nodes with a probability $p$, constructing a random graph $G(V, E)$ following the Erd{\"o}s–R{\'e}nyi model~\cite{Erd1959On}.

\textbf{Multi-populations network.}
Many real complex networks have a modular community structure. To simulate this type of network, we give a straightforward method to generate the multi-populations network ${G_{MP}(V,E)}$. Let ${N_S > 1}$ be the number of population modules in the multi-populations network. We can generate the ${N_S}$-populations network as follows: Firstly, generate ${N_S}$ small-world networks ${P_{s}^{M}(V_s^M,E_s^M)}$ for ${s=1,\cdots,N_S}$, each with the same size (${|V_1^M|}$) as a population module. Then, construct ${m}$ connections (where ${m}$ is a small number) between pairs of population modules to ensure the ${N_S}$-populations network is connected. Thus, ${|V| = N_S \times |V_1^M|}$ and ${|E|}$ is at least equal to ${\sum_{s=1}^{N_S}|E_s^M| + (N_S-1)m}$.

\subsection{More compact topology of BA networks compared to WS networks}

The average distances in small-world networks and scale-free networks differ in their mathematical expressions. The relationships between the network average distance ($\bar{D}$) and the network size ($|V|$) for small-world and scale-free networks~\cite{CH2003} are given by $\bar{D} \sim \ln |V|$ and $\bar{D} \sim \ln(\ln |V|)$, respectively. We tested these relationships through numerical experiments, as shown in Fig.~\ref{fig:S1}. It is clear from the mathematical expressions that scale-free networks have smaller average distances than small-world networks of the same size. Hence, scale-free networks can be termed ultra-small-world networks~\cite{CH2003}. This means that BA networks have a higher compact topology compared to WS networks. We also compared the relationship between network size and average distance for both BA and WS networks. In Fig.~\ref{fig:S2}, we see that the average distance in BA networks is smaller than in WS networks. The ultra-small-world of scale-free networks indicates that these networks are more challenging to compress.

\begin{figure}[htb!]
\centering  
\includegraphics[width=1\linewidth,trim=0 0 0 0,clip]{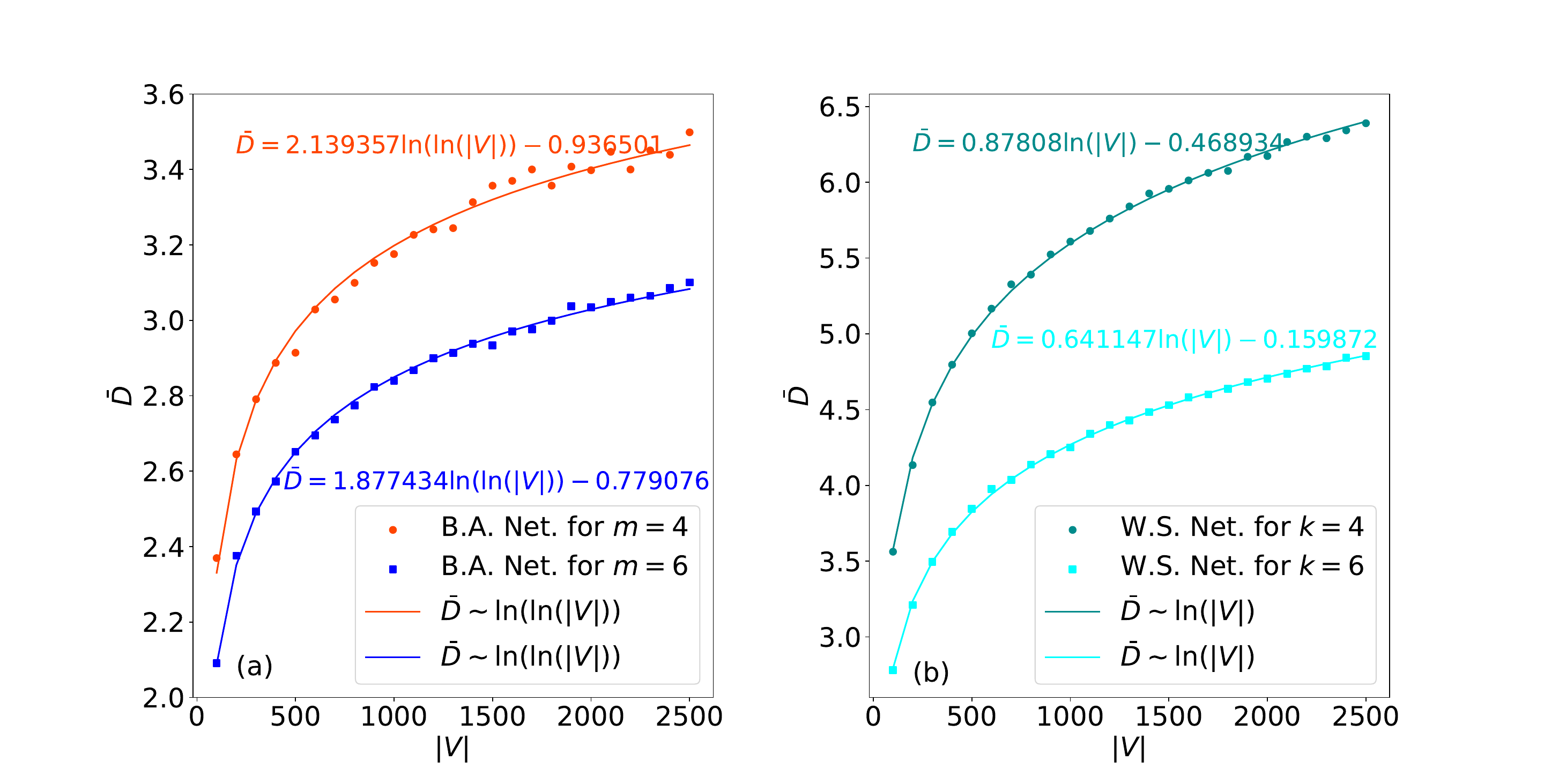}
\caption{The relationship fitting between the average distance and the network size for BA networks and WS networks. Panel (a) shows the double logarithmic relationship $\bar{D} \sim \ln (\ln |V|)$ for BA networks, and panel (b) shows the logarithmic relationship $\bar{D} \sim \ln |V|$ for WS networks. }
\label{fig:S1} 
\end{figure}

\begin{figure}[htb!]
\centering 
\includegraphics[width=1\linewidth,trim=10 5 50 45,clip]{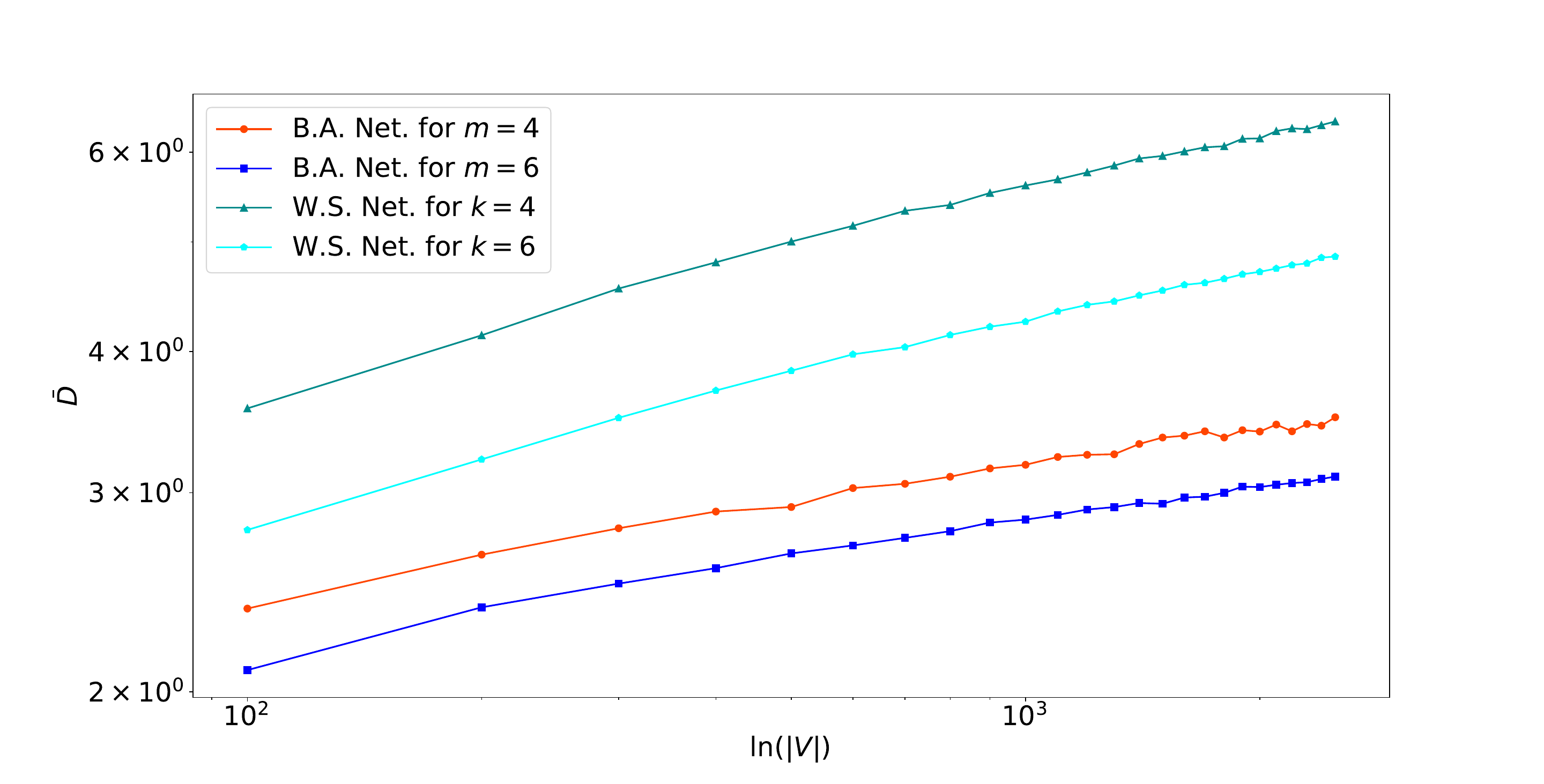}
\caption{The relationship between the average distance and the logarithm of the network size for BA networks and WS networks. }
\label{fig:S2}
\end{figure}

\subsection{Comparison between the effective compression evolution and the random compression}

The effective compression evolution we have designed is expected to be an optimal approach for achieving network tightness through topology compression. In other words, the network produced by the compression process is expected to approximate the limit compression network more closely. To evaluate this, we compare the performance of the effective compression with the method in Ref.~\cite{LYA2023} in various types of networks. 

We show that the effective compression evolution offers outstanding performance in compressing a network to make it more compact. Firstly, for scale-free networks with ultra-small-world characteristics, the compression approach can reduce the average distance, which means that the network is compressed. In stark contrast, the method in Ref.~\cite{LYA2023} fails to facilitate the compression of scale-free networks. Secondly, the compression approach possesses an optimized performance for compressing networks compared to the method in Ref.~\cite{LYA2023} for other types of networks, such as small-world networks. For the BA network alone, the method in Ref.~\cite{LYA2023} cannot represent the topology compression performance. As its verification, we conducted experiments on various types of networks (the topologies of these networks can be seen in Fig.~\ref{fig:S3}).

\begin{figure}[htb!]
\centering  

\includegraphics[width=1\linewidth,trim=125 50 45 60,clip]{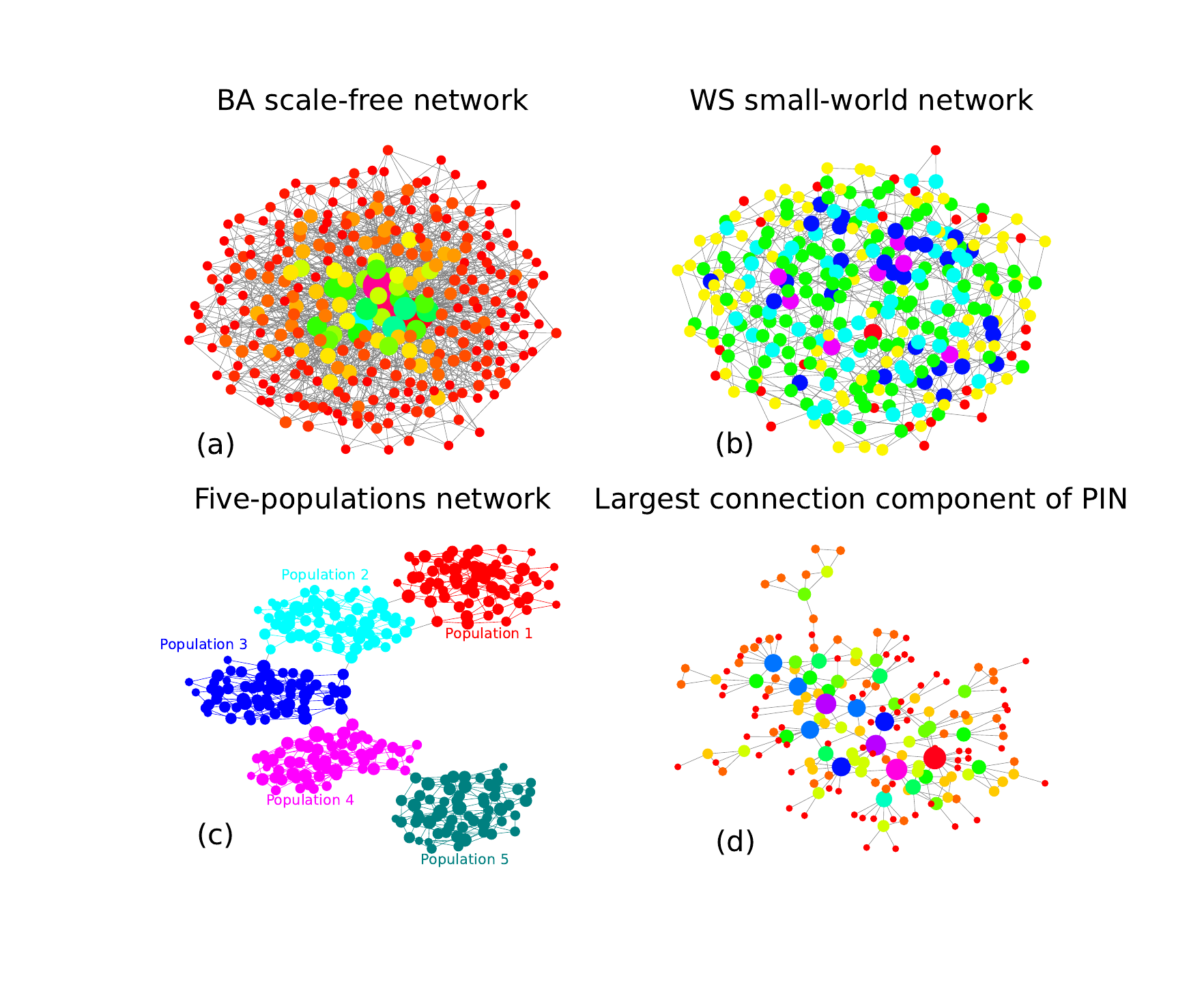}
\caption{The topology of some experimental networks. Panels (a), (b), and (c) show the BA network, the WS network, and the five-populations network with the same network size (300 nodes). Panel (d) shows the largest connected component of a protein interaction network with 175 nodes.}
\label{fig:S3} 
\end{figure}

$(i)$ Setting the parameters as ${|V| = 300}$ and ${m=4}$, we generated a scale-free network using the BA network model. We reshaped the network using the compression approach and the method in Ref.~\cite{LYA2023}, respectively. We then computed the average distance  for all evolved networks. Figure~\ref{fig:S4} illustrates the excellent performance of the compression evolution approach to compress the scale-free network. We observed that applying the compression evolution approach to compress the original network can rapidly reduce the average distance. This means that the topology of the network is rapidly compressed. However, as seen in Fig.~\ref{fig:S4}, instead of 
tightening network, the method in Ref.~\cite{LYA2023} actually makes it worse. This is because BA networks are ultra-small-world networks whose average distances are much smaller than those of WS small-world networks. The topology of BA network is extremely compact and the network is hard to compress, as shown in Fig.~\ref{fig:S4}. Compressing the topology for ultra-compact network, like the ultra-small-world networks, using ordinary compression methods is challenging.

\begin{figure}[htb!]
\centering  

\includegraphics[width=1\linewidth,trim=20 5 60 45,clip]{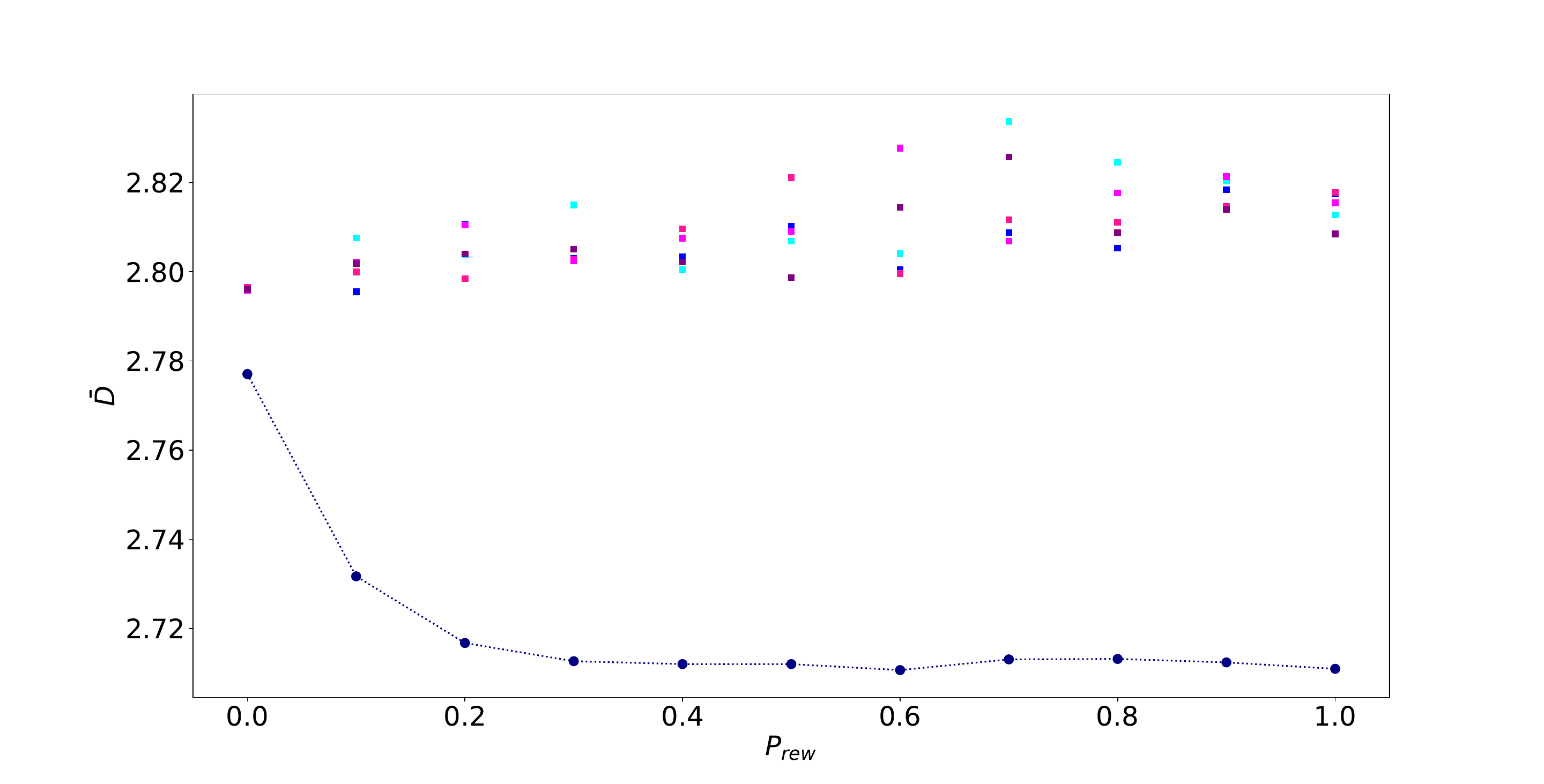}
\caption{The relationship between $\bar{D}$ and the evolution fraction ($P_{rew}$) for BA network. Dotted line represents the simulation results of the compression evolution methods. Scatter plots indicate the results of ten simulations of the method in Ref.~\cite{LYA2023}, with each set of results represented by differently colored markers. } 
\label{fig:S4}
\end{figure}

$(ii)$ We conducted experiments with WS networks with 300 nodes, similarly compressing the WS network. The simulation results are shown in Fig.~\ref{fig:S5}. From Fig.~\ref{fig:S5}, we see that both methods can reduce the average distance of the small-world network. However, the compression evolution approach has a greater performance over the method in Ref.~\cite{LYA2023} to compress the small-world network.

\begin{figure}[htb!]
\centering  

\includegraphics[width=1\linewidth,trim=20 4 60 45,clip]{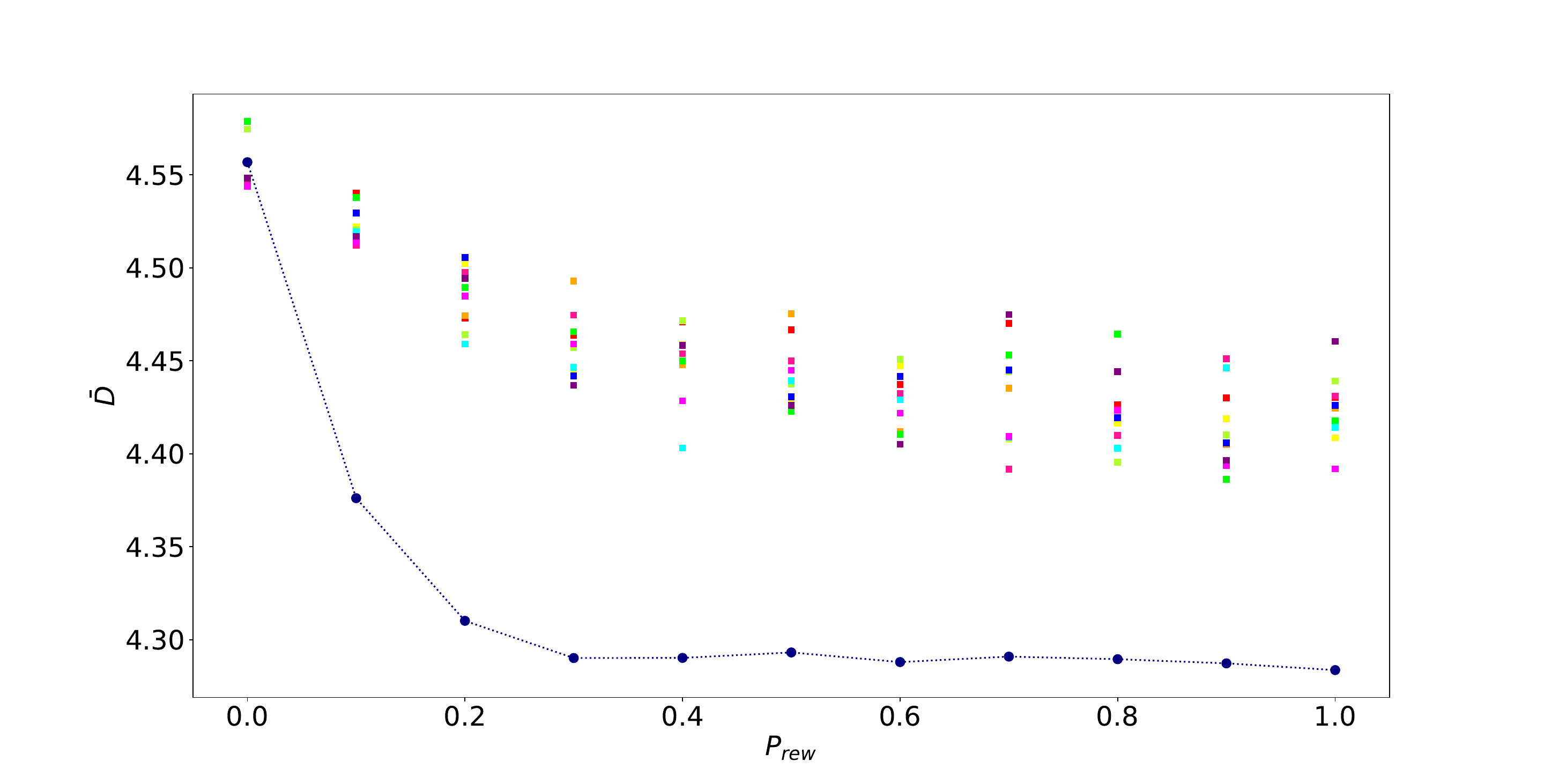}
\caption{The relationship between $\bar{D}$ and $P_{rew}$ for the WS network. Scatters and line are described similarly to Fig.~\ref{fig:S4}.  } 
\label{fig:S5} 
\end{figure}

$(iii)$ We generated a five-populations network for experimental simulations using the above-mentioned generation method for multi-populations network. We first generated five small-world networks of size 60 as population modules and then established two edges between pairs of modules to connect the population network. The topology of the five-populations network is shown in panel (c) of Fig.~\ref{fig:S3}. Figure~\ref{fig:S6} illustrates that both methods can compress the five-populations network topology. However, the compression evolution approach behaves better compression performance.

\begin{figure}[htb!]
\centering  

\includegraphics[width=1\linewidth,trim=20 5 60 45,clip]{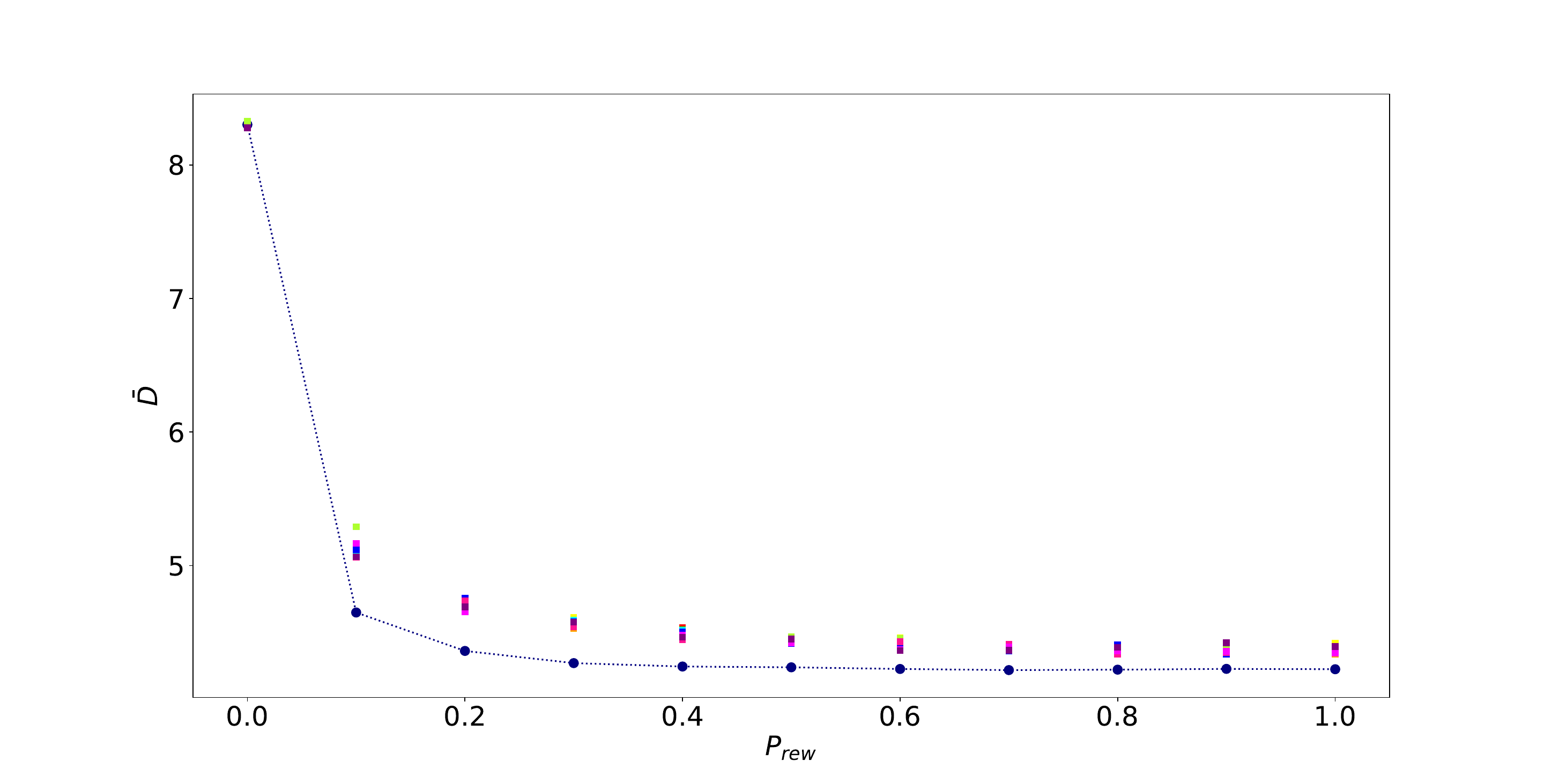}
\caption{The relationship between $\bar{D}$ and $P_{rew}$ for the five-populations network. Scatters and line are described similarly to Fig.~\ref{fig:S4}. }
\label{fig:S6} 
\end{figure}

$(iv)$ As an example of a real-world network, we consider a giant connected component of a protein interaction network (PIN) \cite{BLM2006}. The data can be downloaded with permission from the website \url{http://www.soc.duke.edu/~jmoody77/Prot/index.htm}. We chose the saccharomyces cerevisiae (SC) dataset; the largest connected component of the PIN has 175 nodes (the network topology is shown in panel (d) of Fig.~\ref{fig:S3}). Figure~\ref{fig:S7} illustrates the corresponding experimental results. Both methods can compress the PIN topology. Similarly, the compression evolution approach exhibits superiority.

\begin{figure}[htb!]
\centering 

\includegraphics[width=1\linewidth,trim=20 5 60 45,clip]{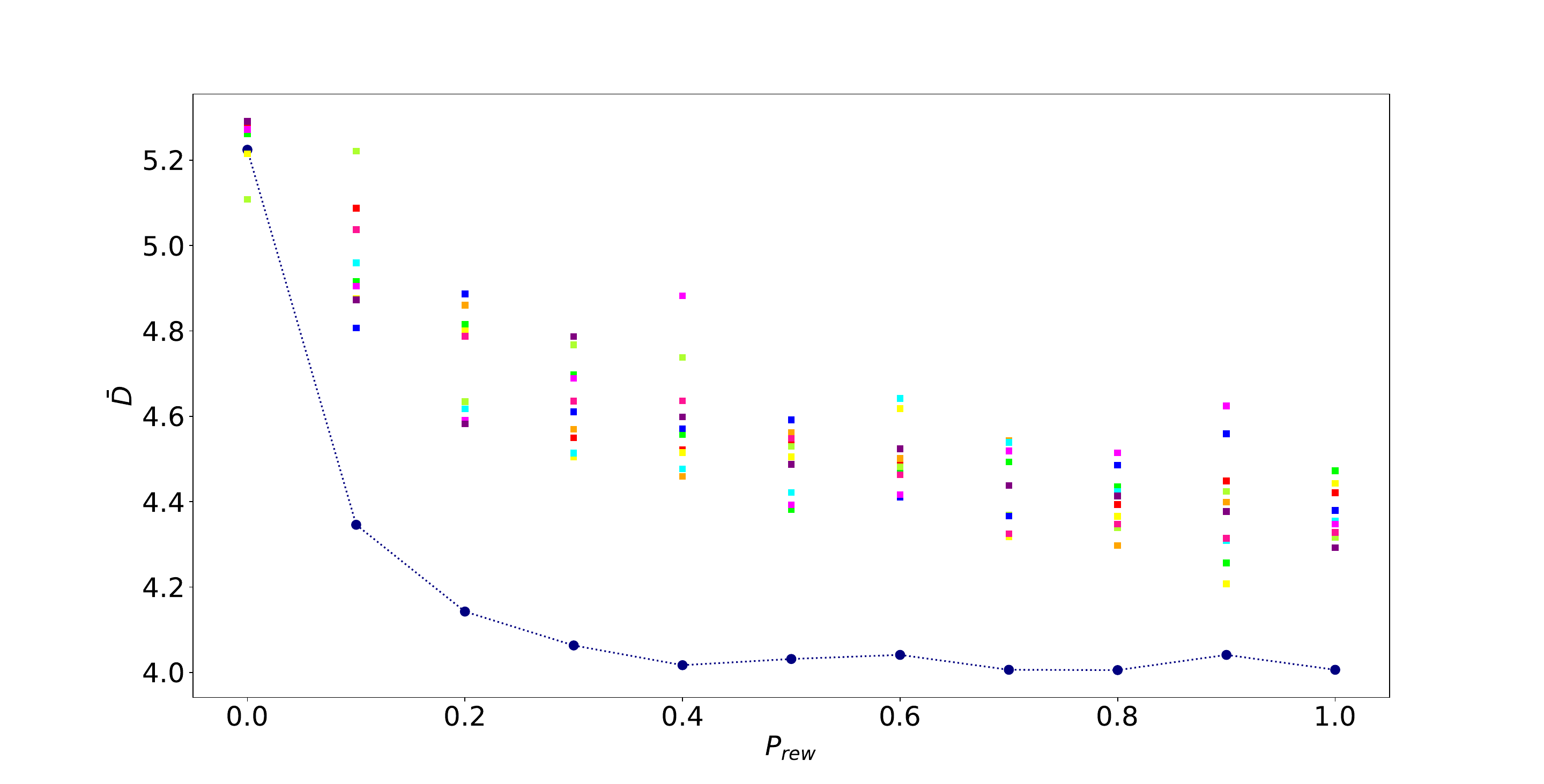}
\caption{The relationship between $\bar{D}$ and $P_{rew}$ for PIN. Scatters and line are described similarly to Fig.~\ref{fig:S4}. } 
\label{fig:S7} 
\end{figure}

\subsection{Effective topological compression of synthetic networks}

We apply the effective compression approach to several widely used synthetic networks, namely the WS, BA, and ER networks. For the ER network, it has the following expected relationships with respect to connection probability $p$,  
\begin{equation*}
|E| = \frac{|V|(|V|-1)p}{2}, \quad \bar{k} = (|V|-1)p, \quad p = \frac{\bar{k}}{(|V|-1)}.  
\label{eqn:er_rela_p}\\
\end{equation*}
For a fair comparison, the connection probability $p$ in the ER model is set to $k/(|V|-1)$, ensuring that the average degree aligns closely with that of the WS networks.

The results for these synthetic networks are shown in Fig~\ref{fig:S8}. We show that the extent to which these networks can be compressed is small. Among the studied networks, the extent of topological compression decreases in the following order: WS network, followed by the ER network, and the lowest is observed in the BA network. In addition, it is shown in Fig~\ref{fig:S8} that these networks essentially approach its compressional limit when the $P_{rew}$ reaches $0.3$.

\begin{figure}[htb!]
\centering 
\includegraphics[width=1\linewidth,trim=74 550 20 52,clip]{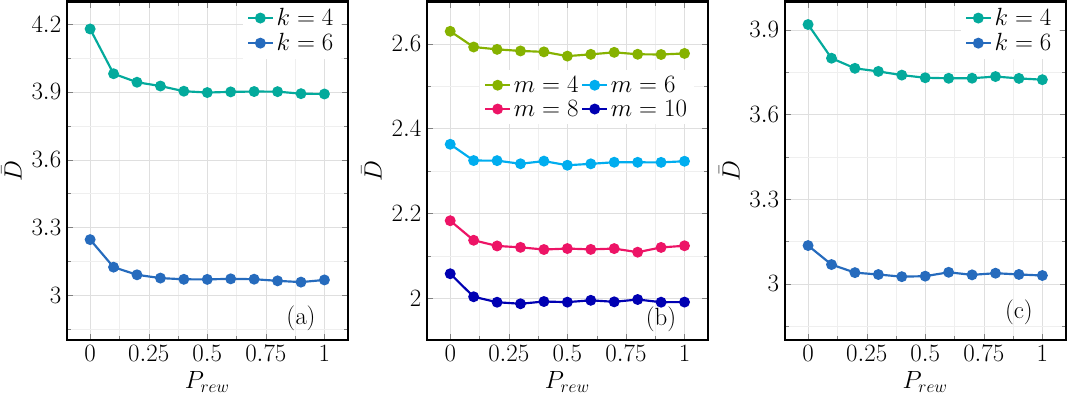}
\caption{
Validation of compression evolution on synthetic networks. 
Panels (a), (b) and (c) show the compression result of WS, BA and ER networks, respectively.  
} 
\label{fig:S8} 
\end{figure}

\subsection{Evaluation of clustering effect for network compression}

For the WS model, as is well known, it is observed that an increase in the edge rewiring probability $P_{rew}$ leads to a decrease in the average clustering coefficient $C$ of the generated network. By setting a smaller $P_{rew}$ for the initial nearest-neighbor network, a small-world network  can be obtained. Consequently, the WS network exhibits a higher value of $C$. However, increasing the rewiring probability $P_{rew}$ results in a decrease in $C$, as the network transitions toward a more random structure.

We perform topological compression evolution on the BA networks, WS networks and ER networks. All experimental results are shown in Fig.~\ref{fig:S10}~(a). From Fig.~\ref{fig:S10}~(a), we also observed the decrease in $C$ as $P_{rew}$ increased. Therefore, both the network topological compression and the $\bar{D}$-reducing transformation, such as the WS model, result in a reduction of  $C$.

\begin{figure}[htb!]
\centering 
\includegraphics[width=1\linewidth,trim=80 30 80 50,clip]{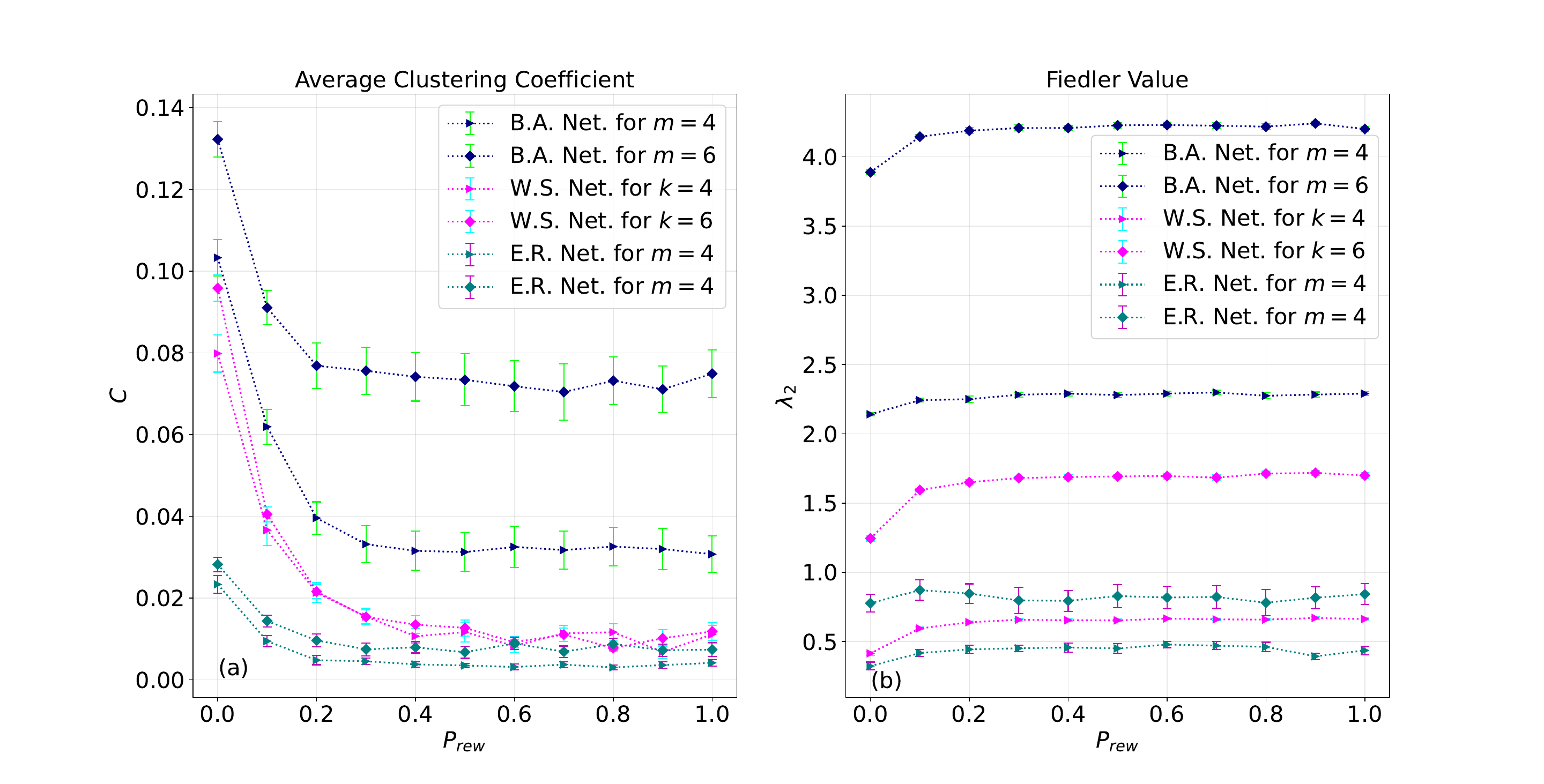}
\caption{Relationship between the average cluster coefficient $C$ and network Fiedler value $\lambda_2$ with $P_{rew}$ for three types of synthetic networks: BA, WS and ER networks.}   
\label{fig:S10} 
\end{figure}

\subsection{Topological compression affects network dynamical properties}

To investigate the effect of network compression on the dynamical properties of the network, we examined variations in the Fiedler values (the real part of the smallest non-zero eigenvalue of network Laplacian matrix) for the BA network, WS network, and ER network. The Fiedler value can reflect the synchronizability~\cite{li2021synchronizability} and diffusion rate~\cite{jiang2024fiedler} of the network. A larger Fiedler value indicates a greater synchronizability for the network. Similarly, the larger the Fiedler value, the larger the diffusion rate of the network.

We compressed all experimental networks by our compression approach and calculated the Fiedler values of these networks, and the experimental results are shown in Fig.~\ref{fig:S10}. Figure~\ref{fig:S10}~(b) shows that the larger the compression fraction $P_{rew}$ of the network, the stronger its synchronizability. Similarly, a larger compression fraction leads to a larger $\lambda_2$ value, indicating a shorter average diffusion time (i.e., higher diffusion rate) for the network.

\section{Conclusion and discussion}

The topological compression of a complex network refers to the process of reshaping its topology such that the average distance of the network decreases, implying that the network becomes more compact. Furthermore, the compression evolution of a complex network does not alter its connectivity and degree distribution. This ensures that network compression process does not lead a connected network to a few connected components. Any process that reshapes network topology under these conditions  is a topological compression of the network. Thus, random rewiring process for some networks is also a network topological compression although it is much less effective than our compression evolution approach.

The topological compression of complex networks is a crucial yet unexplored topic. The topological compressibility, which quantifies how much a network can be further compressed, is an inherent and general property of complex networks. It indicates the level of difficultly to  tighten a network. It also indicates the difference between the original network and its compressible limit.

This work is devoted to investigating an effective compression process, which can be used to characterize the topological compressibility of complex networks. It optimizes the topology of complex networks to make it as compact as possible, hence modeling a general topology compression evolution for networks. A practical value of compact network topology is that a tighter topology (smaller average distance) of a complex network represents a better network structure from the transformation efficiency perspective~\cite{AMS2018}. 

Through rigorous analysis of the structure of complex networks, we have established a theory about upper and lower bounds of network average distance variations resulting from edge removal and addition. Based on this theory, we developed a targeted topological compression evolution process. 

The evolution process has shown significant advantages to compress networks over random rewiring in our testing of general complex networks such as scale-free and small-world networks. For scale-free networks, which are ultra-lower compressible networks, random rewiring is unable to tighten network topology. By contrast, our approach clearly shows a compression of network average distance.

Our compression evolution of networks can adapt to node-constrained networks. Node constraints exist in many real-world networks, e.g., limited radio range in wireless networks as in SF-NANs. They can also be artificially introduced to simplify the computation required for network evolution. Incorporating node constraints into the network evolution, we search a node in a subset of network nodes, which satisfy the node constraint condition, in each requiring step for edge addition. As the search space is much reduced, the computational efficiency is significantly improved in the compression evolution algorithm. 

Given the practical value of compact network topology, we believe that our theory and compression evolution approach provide insights into the design, construction, and optimization of network topology~\cite{noerr2023optimal} or modeling the dynamics of real networks. The following are several examples of designable communication networks that possess potential applications:

(i) Synthetic communications between mammalian cells have been achieved though synthetic receptors~\cite{RWM2016, JOK2018}. It is possible to construct synthetic communication networks between cells in general biological systems~\cite{KSS2022,WDF2007,WCK2008,BLE2012,MBM2022}, and to design specialized cellular communication networks with customized functionalities. New findings in these fields would be significant for biomedical research. 

(ii) Recently, reports have been found on designing synthetic communication networks between artificial cells~\cite{AMG2017}. But this has been less explored~\cite{BEV2020}. 

(iii) Designing gut microbiota and modulating microbial interactions~\cite{GWD2021} is considered as potential new therapies for treating intestinal disorders.

(iv) The brain neuron network is a giant communication network with its network structure being optimized continuously and autonomously for enhanced communication efficiency during the evolution of the species~\cite{AMS2018,LHP2024}. It has been shown that signaling through multiple pathways appears to be a more appropriate communication model for brain networks~\cite{BDK2017}, making the brain more robust and resilient to brain injuries~\cite{KMA2007}. 

(v) Some engineering communication network include Configurable Internet of Things (IoT) networks~\cite{2023Modeling}, Next-generation 5G and 6G wireless communication networks~\cite{yeh2023perspectives,LDY2025}, Quantum communication networks~\cite{cirigliano2024optimal}.

In conclusion, we have developed an effective compression evolution approach to compress network topology as compact as possible. In essence, our approach holds promise for applications across diverse real-world contexts~\cite{AMS2018,noerr2023optimal,RWM2016,JOK2018,KSS2022,WDF2007,WCK2008,BLE2012,MBM2022,AMG2017,BEV2020,GWD2021,BDK2017,KMA2007,2023Modeling,yeh2023perspectives,cirigliano2024optimal}. Our findings also serve as a wellspring of inspiration for tailoring node constraints to optimize network structure and performance, offering useful guidance in this realm, which will be explored in our future work.

\section*{Acknowledgments}

The authors would like to acknowledge Dr Yueming Ding  for discussions on scale-free communication networks of smart grid in engineering. This work was supported in part by the National Natural Science Foundation of China under grant 12371088, and in part by the Australian Research Council (ARC) through the Discovery Projects scheme under grant DP220100580.

\bibliographystyle{unsrt}
\bibliography{Compre_Evo_manu.bib}

\end{document}